\documentclass{article}
\usepackage{amssymb}
\usepackage[english]{babel}
\usepackage{amsthm}
\newtheorem{theorem}{Theorem}
\usepackage[letterpaper,top=2cm,bottom=2cm,left=3cm,right=3cm,marginparwidth=1.75cm]{geometry}

\usepackage{amsmath}
\usepackage{graphicx}
\usepackage[colorlinks=true, allcolors=blue]{hyperref}
\usepackage{authblk}
\usepackage{hyperref}

\providecommand{\keywords}[1]
{
  \small	
  \textbf{\textit{Keywords---}} #1
}

\title{A New Generating Function for Hermite Polynomials}

\author{Manouchehr Amiri}
\affil{Independent Researcher}
\date{}

\affil{manoamiri@gmail.com} 

\begin{document}
\maketitle

\begin{abstract}
This paper presents a new generating function for Hermite polynomials of one variable in the form of $g(x,t)=\sum_{n=0}^{\infty }t^{n}H^{e}_{n}(x)$ and reveals its connection with incomplete gamma function. Recurrence relations of Hermite polynomials are derived by a new generating function. A definition for two-variable Hermite generating function based on this generating function has been proposed.
\end{abstract}


\section{Introduction}

Generating functions as a powerful tool have many applications in the diverse fields of mathematics such as statistics, analysis of algorithms, combinatorics, probability theory, and special functions. Generating function of special functions and polynomials are represented as infinite series over these polynomials in various types \cite{askey1975orthogonal},\cite{brafman1951generating},\cite{cohl2020lectures},\cite{el2006special} and \cite{kruchinin2021method} . For some polynomials this summation does not always contain the terms with integer factorial $\frac{1}{n!}$ such as Laguerre and Legendre polynomials. For the Hermite polynomial, all known generating functions contain factorial terms in the related sum \cite{carlitz1990some}, \cite{fan2015generating} and \cite{thangavelu1993lectures} .In this article we present a new generating function for Hermite polynomials without factorial terms in denominators of the related sum. Important recurrence relations of Hermite polynomials are derived from this new generating function. The connections of this generating function with the error function and the cumulative distribution function are presented. A new integral contour representation for Hermite polynomials is derived. Extension of this generating function to bivariate Hermite polynomials has been shown . 

\section{A New Generating Function}
In this section, a new generating function for one-variable Hermite polynomials is introduced. The recurrence relations and connections to Incomplete gamma function, Error function and cumulative distribution function are discussed in the subsections.

\subsection{New Generating Function for one-variable Hermite Polynomials
}
Known generating function for Hermite polynomials of one variable is represented as:

\begin{equation}
 \sum_{0}^{\infty }H^{e}_{n}(x)\frac{t^{n}}{n!}=e^{xt-\frac{1}{2}t^{2}}  
\end{equation}\

Where $H^{e}_{n}(x)=He_{n}(x)$ is denoted as probabilistic Hermite polynomials. This generating function contains a factorial term $\frac{1}{n!}$. We present a new generating function without the factorial term in the related sum.  
Probabilistic Hermite polynomials $H^{e}_{n}(x)$ are presented as:

\begin{equation}
 H^{e}_{n}(x)=e^{\frac{-D^{2}}{2}}x^{n}   
\end{equation}\

Where $D=\frac{d }{dx}$. We apply this operator equation to obtain a generating function for $H^{e}_{n}(x)$. 

\begin{theorem}  Let $  x\in \mathbb{R},t\in \mathbb{R} -\left\{ 0 \right\} ,n\in \mathbb{N} \cup \left\{ 0 \right\} $

A generating function for Hermite polynomials is:
         
     \begin{equation}
     g(x,t)=\sum_{n=0}^{\infty }t^{n}H^{e}_{n}(x)=\frac{1}{\sqrt{2t^{2}}}e^{\frac{(1-xt)^{2}}{2t^{2}}}\Gamma(\frac{1}{2},\frac{(1-xt)^{2}}{2t^{2}})
    \end{equation}\  

\end{theorem}
    
    \begin{proof}

    \textup{Acting operator $O=e^{\frac{-\partial^{2}_{x}}{2}}$ from equation (2) on the series $\sum_{n=0}^{\infty }t^{n}x^{n}=\frac{1}{1-xt}$ , yields a generating function for $H_{n}^{e}(x)$:} 

\begin{equation}
    g(x,t)=e^{\frac{-\partial^{2}_{x}}{2}}\sum_{n=0}^{\infty }t^{n}x^{n}=\sum_{n=0}^{\infty }t^{n}H_{n}^{e}(x)=e^{\frac{-\partial^{2}_{x}}{2}}\left( \frac{1}{1-xt} \right)
    \end{equation}\
    The convergence of $\sum_{n=0}^{\infty }t^{n}x^{n}=\frac{1}{1-xt}$ requires $\left| xt \right| < 1$ or $0<1-xt<2$

    \textup{Expanding the right side of equation (4) gives:}

    \begin{equation}
    g(x,t)=\sum_{j=0}^{\infty }(-1)^{j}\frac{t^{2j}(1-xt)^{-2j-1}(2j)!}{2^{j}j!}
    \end{equation}\

    \textup{With respect to the identity $\frac{(2j)!}{2^{j}j!}=\frac{\sqrt{\pi }(-2)^{j} }{\Gamma(\frac{1}{2}-j)}$ and denoting $y=t^{-1}(1-xt)$  we have:}

    \begin{equation}
    g(x,t)=\frac{\sqrt{\pi}}{(1-xt)}\sum_{j=0}^{\infty }
   (\frac{y^{2}}{2})^{-j}\Gamma(\frac{1}{2}-j)^{-1}
    \end{equation}\

   By the identity for $ n-th $ derivative of $ z^{s}$ when $ n\in \mathbb{N}$ we get
    
    \begin{equation}
    \frac{d^{n} }{dz^{n}} z^{s}=\frac{\Gamma(s+1)}{\Gamma(s+1-n)}z^{s-n}
    \end{equation}\

    \textup{For $s=-\frac{1}{2}$  we have:}
    \begin{equation}
    \frac{d^{n} }{dz^{n}} z^{-\frac{1}{2}}=\frac{\Gamma(\frac{1}{2})}{\Gamma(\frac{1}{2}-n)} z^{-\frac{1}{2}-n} 
    \end{equation}\

    \begin{equation}
    \sum_{n=1}^{\infty }\frac{d ^{n}}{dz^{n}} z^{-\frac{1}{2}}=\Gamma(\frac{1}{2}) z^{-\frac{1}{2}}\sum_{n=1}^{\infty }\frac{z^{-n}}{\Gamma(\frac{1}{2}-n)}
    \end{equation}\ 

    \textup{The strict condition for convergence of this formal series is $\left| z \right| > 1$ . Rewriting equation (9) by replacing $j=n$ on the right side, yields:}

    \begin{equation}
    z^{-\frac{1}{2}}+\sum_{n=1}^{\infty }\frac{d ^{n}}{dz^{n}} z^{-\frac{1}{2}}=\Gamma(\frac{1}{2}) z^{-\frac{1}{2}}\sum_{j=0}^{\infty }\frac{z^{-j}}{\Gamma(\frac{1}{2}-j)}
    \end{equation}\

    By changing the variable $\frac{y^{2}}{2}=z $ and with respect to $1-xt>0$, for $z^{-\frac{1}{2}}$  we have:

    \begin{equation}
    z^{-\frac{1}{2}}=\sqrt{2} \frac{\left|t\right|}{1-xt}= \frac{\sqrt{2t^{2}}}{1-xt} 
    \end{equation}\

    By this equation and comparing equation (10) with (6)   and applying the identity $\Gamma(\frac{1}{2})=\sqrt{\pi}$, equation (10) becomes:   
    
    \begin{equation}
    z^{-\frac{1}{2}}+\sum_{n=1}^{\infty }\frac{d ^{n}}{dz^{n}} z^{-\frac{1}{2}}=\sqrt{2t^{2}} g(x,t)
    \end{equation}\

    \textup{ The left side of equation (12) can be calculated as follows:}
    
    \begin{equation}
   z^{-\frac{1}{2}}+\sum_{n=1}^{\infty }\frac{d ^{n}}{dz^{n}} z^{-\frac{1}{2}}=\frac{1}{(1-\frac{d }{dz})} z^{-\frac{1}{2}}
    \end{equation}\

    \textup{Then by the following integral identity we get}

    \begin{equation}
    \frac{1}{(1-\frac{d }{dz})} z^{-\frac{1}{2}}= -e^{z}\int_{}^{}e^{-z} z^{-\frac{1}{2}}dz
    =-e^{z}\left[ -\Gamma(\frac{1}{2},z) \right]
    \end{equation}\

    \textup{ By Calculating the integral, the equation (12) becomes:}

    \begin{equation}
    g(x,t)=\frac{1}{\sqrt{2t^{2}}}e^{z}\Gamma(\frac{1}{2},z)
    \end{equation}\

    \textup{ The term $\Gamma(\frac{1}{2},z)$ on the right side is incomplete gamma function. By writing $z$ in terms of $t$ and $x$ we obtain:}

    \begin{equation}
    g(x,t)=\sum_{n=0}^{\infty }t^{n}H^{e}_{n}(x)=\frac{1}{\sqrt{2t^{2}}}e^{\frac{(1-xt)^{2}}{2t^{2}}}\Gamma(\frac{1}{2},\frac{(1-xt)^{2}}{2t^{2}})
    \end{equation}\  
    \end{proof}

    \textup{This is a new generating function for Hermite polynomials. As a test for validation of equation (16) at the limit $t\to 0$, by knowing the limit:}
    
    \begin{equation}
    \lim_{t\to 0} g(x,t)=\lim_{t\to 0}\sum_{n=0}^{\infty }t^{n}H^{e}_{n}(x)=1
    \end{equation}\  
    
    \textup{We can verify the limit of the right side of (16) when $t\to 0$ .} 

    \begin{equation}
   \lim_{t\to 0} g(x,t)=\lim_{t\to 0}\sum_{n=0}^{\infty }t^{n}H^{e}_{n}(x)=\lim_{t\to 0}\frac{1}{\sqrt{2t^{2}}}e^{\frac{(1-xt)^{2}}{2t^{2}}}\Gamma(\frac{1}{2},\frac{(1-xt)^{2}}{2t^{2}})=1
    \end{equation}\
  
   \textup{For validation of equation (18), this limit was calculated by online calculator for some values of $\lim_{t\to 0} g(x,t)$. Calculations of equation (15) for $z=1200$ and $z=2000$ while $t$ is small and $x=1$, result in 0.994 and 0.997 for $g(x,t)$  respectively. This intuitively verifies the limit in equation (18). }\\
   
   \textbf{Remark} 
   
   Due to the term $\sqrt{2}\left|t\right|=\sqrt{2t^{2}}$ in the generating function (16), the derivative of the generating function could be calculated for $t>0$ and $t<0$ separately. The results of both are the same for derivation of recurrence relations that are discussed in the next section, thus the recurrence relations have been derived for $t>0$.

\subsection{Derivation of Recurrence Relations}
   
   \textup{In this section we derive the Recurrence relations of Hermite polynomials by generating function introduced in the previous section. Considering the remark in previous section, we calculate the partial derivatives of the right side of the generating function $g(x,t)$ in equation (16) respect to $t$ and $x$ respectively. Due to to the derivative relation for $\Gamma(\frac{1}{2},z)$ :}

   \begin{equation}
    \frac{d }{dz}\Gamma\left( \frac{1}{2},z \right)=-z^{-\frac{1}{2}}e^{-z}
    \end{equation}\

    \textup{ Derivative of right side of equation (16) respect to $t$ gives:}

   \begin{equation}
    \frac{\partial }{\partial t}g(x,t)=-\frac{1}{t}g(x,t)-\frac{1-xt}{t^{3}}g(x,t)+\frac{1}{t^{3}}
    \end{equation}\

    \textup{ Derivative of same equation respect to $x$ yields:}

    \begin{equation}
    \frac{\partial }{\partial x}g(x,t)=-\frac{1-xt}{t}g(x,t)+\frac{1}{t}
    \end{equation}\

    Substitution of $g(x,t)$ in equation (20) by the summation definition of equation (3) yields the following:

    \begin{equation}
    \frac{\partial }{\partial x}\sum_{n=0}^{\infty }t^{n}H^{e}_{n}(x)=-\frac{1-xt}{t}\sum_{n=0}^{\infty }t^{n}H^{e}_{n}(x)+\frac{1}{t}
    \end{equation}\

    Thus, we get:

    \begin{equation}
    \sum_{n=0}^{\infty }t^{n}\frac{d }{dx}H^{e}_{n}(x)=-\frac{1}{t}\sum_{n=0}^{\infty }t^{n}H^{e}_{n}(x)+x\sum_{n=0}^{\infty }t^{n}H^{e}_{n}(x)+\frac{1}{t}
    \end{equation}\

    \begin{equation}
    \sum_{n=0}^{\infty }t^{n}\frac{d }{dx}H^{e}_{n}(x)=-\sum_{n=0}^{\infty }t^{n-1}H^{e}_{n}(x)+x\sum_{n=0}^{\infty }t^{n}H^{e}_{n}(x)+\frac{1}{t}
    \end{equation}\

    \begin{equation}
    \sum_{n=0}^{\infty }t^{n}\frac{d }{dx}H^{e}_{n}(x)=-\frac{1}{t}-\sum_{n=0}^{\infty }t^{n}H^{e}_{n+1}(x)+x\sum_{n=0}^{\infty }t^{n}H^{e}_{n}(x)+\frac{1}{t}
    \end{equation}\

    Equating the coefficients of $t^{n}$ on both sides gives:

    \begin{equation}
    \frac{d }{dx}H^{e}_{n}(x)=-H^{e}_{n+1}(x)+xH^{e}_{n}(x)
    \end{equation}\
 This is one of the recurrence relation of Hermite polynomials.

 Substitution of $g(x,t)$ in equation (20) by the summation definition of equation (3) yields the following:

 \begin{equation}
    \frac{\partial }{\partial t}\sum_{n=0}^{\infty }t^{n}H^{e}_{n}(x)=-\frac{1}{t}\sum_{n=0}^{\infty }t^{n}H^{e}_{n}(x)-\frac{1-xt}{t^{3}}\sum_{n=0}^{\infty }t^{n}H^{e}_{n}(x)+\frac{1}{t^{3}}
    \end{equation}\

\begin{equation}
    \sum_{n=0}^{\infty }nt^{n-1}H^{e}_{n}(x)=-\sum_{n=0}^{\infty }t^{n-1}H^{e}_{n}(x)-\frac{1}{t^{3}}\sum_{n=0}^{\infty }t^{n}H^{e}_{n}(x)+ \frac{x}{t^{2}}\sum_{n=0}^{\infty }t^{n}H^{e}_{n}(x)+\frac{1}{t^{3}}
    \end{equation}\

 \begin{equation}
    \sum_{n=0}^{\infty }nt^{n-1}H^{e}_{n}(x)=-\sum_{n=0}^{\infty }t^{n-1}H^{e}_{n}(x)-\frac{1}{t^{3}}-\sum_{n=1}^{\infty }t^{n-3}H^{e}_{n}(x)+ x\sum_{n=0}^{\infty }t^{n-2}H^{e}_{n}(x)+\frac{1}{t^{3}}
    \end{equation}\   
    
    \textup{Multiplying both sides of Equations (29) by $t^2$ gives:}

  \begin{equation}
    \sum_{n=0}^{\infty }(n+1)t^{n+1}H^{e}_{n}(x)=-\sum_{n=1}^{\infty }t^{n-1}H^{e}_{n}(x)+ x\sum_{n=0}^{\infty }t^{n}H^{e}_{n}(x)
    \end{equation}\  

    Equating the coefficients of $t^{n}$ on both sides gives:

   \begin{equation}
    nH^{e}_{n-1}(x)=-H^{e}_{n+1}(x)+ xH^{e}_{n}(x)
    \end{equation}\  

    This is another recurrence relation of Hermite polynomials.

    \subsection{Associated Partial Differential Equation of Generating Function}

    \textup{Multiplying both sides of Equations (20) by $t^2$ gives:}

    \begin{equation}
    t^{2}\frac{\partial }{\partial t}g(x,t)=-tg(x,t)-\frac{1-xt}{t}g(x,t)+\frac{1}{t}
    \end{equation}\

    \textup{By comparing the equations (21) and (32) we have:}

    \begin{equation}
    t^{2}\frac{\partial }{\partial t}g(x,t)=\frac{\partial }{\partial x}g(x,t)-tg(x,t)
    \end{equation}\

    \textup{Performing the same derivatives on the generating function that has been defined as a sum in equation (3), yields the same relation in equation (33):}

    \begin{equation}
    t^{2}\frac{\partial }{\partial t}\sum_{n=0}^{\infty }t^{n}H^{e}_{n}(x)=\frac{\partial }{\partial x}\sum_{n=0}^{\infty }t^{n}H^{e}_{n}(x)-t\sum_{n=0}^{\infty }t^{n}H^{e}_{n}(x)
    \end{equation}\

    \textup{This verifies the generating function equation 
    introduced in equation (16).}

    \subsection{Solutions to Partial Differential Equations Derived from Generating Functions}

    We show that the solutions to partial differential equation (32) encompass the form of the generating function introduced in equation (3).

    The auxiliary system of equations for first order partial differential equation introduced in equation (33) is presented as:

    \begin{equation}
   \frac{d t}{t^{2}}=-dx=-\frac{d g(x,t)}{tg(x,t)}
    \end{equation}\

    Integrating the first two with the t auxiliary equations yields two independent first integrals as follows:

    \begin{equation}
   \psi_{1}=-\frac{1}{t}+x=c_{1} \quad\quad \psi_{2}=g(x,t)t=c_{2}
    \end{equation}\ 

    Thus the integral of equation (33) reads as:

    \begin{equation}
   \Phi(\psi_{1},\psi_{2})=0
    \end{equation}\
    Where $\Phi$ is an arbitrary function.
    As a solution for $g(x,t)$ we may obtain:

    \begin{equation}
   tg(x,t)=\varphi(-\frac{1}{t}+x)
    \end{equation}\

 Therefor the general solution may read as:   

 \begin{equation}
   g(x,t)=\frac{1}{t}\varphi(-\frac{1}{t}+x)
 \end{equation}\   

This solution is compatible with the form of new generating function in equation (3).  
\\    

The differential equation (33) associated to new generating function, differs from the associated differential equation of presented generating function  in equation (1).  Repeating the same derivatives respect to $t$ and $x$ for generating function presented in equation (1):

\begin{equation}
   G(x,t)=\sum_{0}^{\infty }H^{e}_{n}(x)\frac{t^{n}}{n!}=e^{xt-\frac{1}{2}t^{2}}
   \end{equation}\

yields the relations:

\begin{equation}
  \frac{\partial }{\partial x}G(x,t)=tG(x,t) ,\quad \frac{\partial }{\partial t}G(x,t)=(x-t)G(x,t)
   \end{equation}\

 These equations result in the partial differential equation:

 \begin{equation}
  \frac{\partial }{\partial x}G(x,t)+\frac{\partial }{\partial t}G(x,t)=xG(x,t) 
    \end{equation}\
 This differs from the differential equation presented for $g(x,t)$ in equation (33).   

 The solution to this partial differential equation through its auxiliary equations reads as:

\begin{equation}
   \frac{dx}{1}=\frac{dt}{1}=\frac{dG(x,t)}{xG(x,t)}
   \end{equation}\

\begin{equation}
   \psi_{1}=x-t=c_{1} , \quad \psi_{2}=G(x,t)e^{-\frac{x^{2}}{2}}=c_{2}
    \end{equation}\ 

Thus for general solution $\Phi(\psi_{1},\psi_{2})=0
    $  we have:
    
\begin{equation}
   G(x,t)e^{-\frac{x^{2}}{2}}= \varphi(x-t)
   \end{equation}\ 

By choosing:

\begin{equation}
   \varphi(x-t)=e^{-\frac{(x-t)^{2}}{2}}
    \end{equation}\

The solution reads as:    

\begin{equation}
  G(x,t)e^{-\frac{x^{2}}{2}}=e^{-\frac{(x-t)^{2}}{2}}
   \end{equation}\ 
Or:
\begin{equation}
   G(x,t)=e^{xt-\frac{1}{2}t^{2}}
   \end{equation}\
This is the generating function in equation (1).

\subsection{Relation to Error Function and Cumulative Distribution Function}  

Respect to equation (16) and the identity for complementary Error function:

\begin{equation}
   erfc(x)=\frac{1}{\sqrt{\pi}}\Gamma(\frac{1}{2},x^{2})
    \end{equation}\ 

We obtain:
\begin{equation}
g(x,t)=\sum_{n=0}^{\infty }t^{n}H^{e}_{n}(x)=\frac{\sqrt{\pi}}{\sqrt{2}\left | t \right |}e^{\frac{(1-xt)^{2}}{2t^{2}}}erfc(\frac{1-xt}{\sqrt{2}\left | t \right |})
\end{equation}\ 

Respect to relations between cumulative probability of a normal distribution (CDF) and error function we have:
\begin{equation}
  \Phi(x,\mu,\sigma)=1-\frac{1}{2}\left( erfc\left( \frac{x-\mu}{\sqrt{2}\sigma} \right) \right)
    \end{equation}\ 

By the identity $ erfc(-x)=2-erfc(x) $ we get:

\begin{equation}
  \Phi(x,\mu,\sigma)=\frac{1}{2}\left( erfc\left( \frac{\mu-x}{\sqrt{2}\sigma} \right) \right)
    \end{equation}\ 
    
Let $\sigma=1$ and $\mu =\frac{1}{t}$, thus by equation (50) we obtain:
\begin{equation}
  \sum_{n=0}^{\infty }t^{n}H^{e}_{n}(x)=\sum_{n=0}^{\infty }\left( \frac{1}{\mu} \right)^{n}H^{e}_{n}(x)=\sum_{n=0}^{\infty }\frac{H^{e}_{n}(x)}{\mu^{n}}=\sqrt{2\pi}\mu\Phi(x,\mu,1)
    \end{equation}\ 

The last equation implies:

\begin{equation}
  \sum_{n=0}^{\infty }\frac{H^{e}_{n}(x)}{\mu^{n+1}}=\sqrt{2\pi}\Phi(x,\mu,1)
    \end{equation}\ 

With the identity ${\mathbb{E}\left[ H^{e}_{n}(x) \right]}=\mu^{n} $, equation (54) becomes:

\begin{equation}
  \sum_{n=0}^{\infty }\frac{{\mathbb{E}\left[ H^{e}_{n}(x) \right]}}{\mu^{n+1}}=\sum_{n=0}^{\infty }\frac{1}{\mu}=\sqrt{2\pi}{\mathbb{E}\left[\Phi(x,\mu,1) \right]}=\infty 
    \end{equation}\ 

Squaring both sides of equation (54) yields:

\begin{equation}
  \sum_{m,n=0}^{\infty }\frac{H^{e}_{n}(x)H^{e}_{m}(x)}{\mu^{n+m+2}}=2\pi\Phi^{2}(x,\mu,1)
    \end{equation}\ 

The equations(50),(54),(55) and (56) are new relations which hold right for Hermite polynomials.    
    
\section{New Generating Function for Bivariate Hermite Polynomials}

\begin{theorem} Let $ x,y\in \mathbb{R}, s,t\in \mathbb{R}-\left\{ 0 \right\} , m,n\in \mathbb{N}\cup \left\{0 \right\}$

 The generating functions for product of Hermite polynomials and Bivariate Hermite polynomials are respectively as follows:
 
 \begin{equation}
     g(s,t,x,y)=\frac{1}{2\left | ts \right |}e^{\frac{(1-xt)^{2}}{2t^{2}}}e^{\frac{(1-ys)^{2}}{2s^{2}}}\Gamma(\frac{1}{2},\frac{(1-xt)^{2}}{2t^{2}})\Gamma(\frac{1}{2},\frac{(1-ys)^{2}}{2s^{2}})    
 \end{equation}

 \begin{equation}
      g(x,y,t,s)=\sum_{n=0}^{\infty }\sum_{m=0}^{\infty }t^{n}s^{m}H_{n,m}^{e}(\mathrm{M};x,y)=\frac{1}{2\left | t's' \right |}e^{\frac{(1-xt')^{2}}{2t'^{2}}}e^{\frac{(1-ys')^{2}}{2s'^{2}}}\Gamma(\frac{1}{2},\frac{(1-xt')^{2}}{2t'^{2}})\Gamma(\frac{1}{2},\frac{(1-ys')^{2}}{2s'^{2}}) 
 \end{equation}

 \end{theorem}

 \begin{proof}
 Respect to equation (3) we have:

\begin{equation}
     g(x,y,t,s)=g(x,t)g(y,s)=\sum_{n,m=0}^{\infty }t^{n}s^{m}H^{e}_{n}(x)H^{e}_{m}(y)=\frac{1}{2\left | ts \right |}e^{\frac{(1-xt)^{2}}{2t^{2}}}e^{\frac{(1-ys)^{2}}{2s^{2}}}\Gamma(\frac{1}{2},\frac{(1-xt)^{2}}{2t^{2}})\Gamma(\frac{1}{2},\frac{(1-ys)^{2}}{2s^{2}}) 
 \end{equation}

 This is a simple product of two Hermite polynomials. For a general definition of the 2-dimensional Hermite polynomial which is based on the generating function defined in equation (3), first we introduce the 2-dimensional Hermite polynomials as defined in \cite{wunsche2001hermite} :

 \begin{equation}
     g(s,t,x,y)=\sum_{n=0}^{\infty }\sum_{m=0}^{\infty }t^{n}s^{m}H_{n,m}^{e}(\mathrm{M};x,y) 
 \end{equation}

 Where $\mathrm{M}$ is a matrix with real entries:
\begin{equation}
  \mathrm{M}= \begin{bmatrix}
a & b \\
c & d
\end{bmatrix}  
 \end{equation}

 That transforms the vector  $(t,s)$ to $(t',s')$ from the right side:
 
 \begin{equation}
  (t',s')=(t,s)\mathrm{M}\Rightarrow t'=at+cs \quad s'=bt+ds 
 \end{equation}
 
 Thus, the new proposed 2-dimensional generating function is defined by the equation:

  \begin{equation}
      g(x,y,t,s)=\sum_{n=0}^{\infty }\sum_{m=0}^{\infty }t^{n}s^{m}H_{n,m}^{e}(\mathrm{M};x,y)=\frac{1}{2\left | t's' \right |}e^{\frac{(1-xt')^{2}}{2t'^{2}}}e^{\frac{(1-ys')^{2}}{2s'^{2}}}\Gamma(\frac{1}{2},\frac{(1-xt')^{2}}{2t'^{2}})\Gamma(\frac{1}{2},\frac{(1-ys')^{2}}{2s'^{2}}) 
 \end{equation}

 For special case with $\mathrm{M}=\textbf{I}$ we have, $s'=s$ and $t'=t $, thus the equation (63) reduces to the simpler form in equation (59):
\begin{equation}
 g(x,y,t,s)=g(x,t)g(y,s)
 \end{equation}
 \end{proof}

 \section{Contour Integral Representation}

 In this section, the contour integral representation for new generating function is derived. Taking into account the formal generating function introduced in equation (1), the contour integral representation of Hermite polynomials for $t,x\in \mathbb{C}$ reads as \cite{abramowitz1988handbook}:

\begin{equation}
 H^{e}_{n}(x)=\frac{n!}{2\pi i}\oint_{C}\frac{e^{xt-\frac{1}{2}t^{2}}}{t^{n+1}}dt
 \end{equation}
Where $C$ is a simple closed contour in complex plane encircles $t=0$ as the sole singular point of the integrand.
The numerator of the integrand is recognized as the ordinary generating function for Hermite polynomial that was introduced in equation (1). We introduce an alternative contour integral for Hermite polynomials based on new generating function in equation (16).

 \begin{theorem}
     Let $ x\in \mathbb{C},t\in \mathbb{C}-\left\{ 0 \right\}$. A contour integral representation for Hermite polynomials is :
 \begin{equation}
 H^{e}_{n}(x)=\frac{1}{2\sqrt{2}\pi i}\oint_{\gamma}\frac{e^{\frac{(1-xt)^{2}}{2t^{2}}}\Gamma(\frac{1}{2},\frac{(1-xt)^{2}}{2t^{2}})}{t^{n+2}}dt
 \end{equation} 
 Where $\gamma$ is defined as the simple closed curve (contour) with center at $t=0$ .
 \end{theorem}

 \begin{proof}

 By taking $n-th$ derivatives with respect to $t$ for the generating function in Equation (16) and calculating it at $t\rightarrow 0$ we obtain:
 
\begin{equation}
 n!H^{e}_{n}(x)=\lim_{t\to 0}\frac{d^{n}}{dt^{n}}g(x,t)=\lim_{t\to 0}\frac{d^{n}}{dt^{n}}\left[\frac{1}{\sqrt{2t^{2}}}e^{\frac{(1-xt)^{2}}{2t^{2}}}\Gamma(\frac{1}{2},\frac{(1-xt)^{2}}{2t^{2}})  \right]
 \end{equation}

 Let $\gamma$ be a simple closed contour in complex plane . If the complex function $f(z)$ is analytic on a region containing $\gamma$ and its interior, the Cauchy integral for $n-th$ derivative of  $f(z)$ at a point $z_{0}$ inside the contour $\gamma$ is presented as \cite{weber2003essential}:

 \begin{equation}
  f^{\left( n \right)}(z_{0})=\frac{n!}{2\pi i}\oint_{\gamma}\frac{f(z)}{\left( z-z_{0} \right)^{n+1}}dz
 \end{equation}

 Where $\gamma$ is a simple closed contour containing $z_{0}$. In the case where $z=0$ is a singular point of $f(z)$, the integral could be written at the limit ${z_{0}\to 0}$ as follows :

\begin{equation}
  \lim_{z_{0}\to 0}f^{\left( n \right)}(z_{0})=\frac{n!}{2\pi i}\lim_{z_{0}\to 0}\oint_{\gamma}\frac{f(z)}{\left( z-z_{0} \right)^{n+1}}dz=\oint_{\gamma}\frac{f(z)}{z^{n+1}}dz
 \end{equation}

 Thus respect to the Taylor expansion coefficients of generating 
 function on the right side of (67) and replacing $f(z)$ by $g(x,t)$, and  $z$ by $t$ for fixed $x$, the Cauchy contour integral formula for $n-th$ derivative, yields:

\begin{equation}
 \lim_{t\to 0}\frac{d^{n}}{dt^{n}}g(x,t)=\frac{n!}{2\sqrt{2}\pi i}\oint_{\gamma}\frac{e^{\frac{(1-xt)^{2}}{2t^{2}}}\Gamma(\frac{1}{2},\frac{(1-xt)^{2}}{2t^{2}})}{t^{n+2}} dt
 \end{equation}

 Thus, with respect to equation (67) we have:

 \begin{equation}
 H^{e}_{n}(x)=\frac{1}{2\sqrt{2}\pi i}\oint_{\gamma}\frac{e^{\frac{(1-xt)^{2}}{2t^{2}}}\Gamma(\frac{1}{2},\frac{(1-xt)^{2}}{2t^{2}})}{t^{n+2}}dt
 \end{equation}

 \end{proof}
 This is a new contour integral representation for Hermite polynomial over the complex field $\mathbb{C}$.

\section{Conclusion}

Applying the operational method to derive a new generating function for Hermite polynomials reveals its connection with incomplete gamma function and a sum without integer factorials, i.e. $n!$. The close relation between incomplete gamma function and error function results in new relations between this functions and cumulative distribution function and Hermite generating function. A definition for two-variable Hermite generating function based on the new generating function has been proposed. An alternative contour integral representation for Hermite polynomials has been derived. By more investigations, this new generating function may lead to more hidden relations between Hermite polynomial and other special functions.

\bibliographystyle{abbrv}
\bibliography{sample}

\begin{thebibliography}{10}

\bibitem{abramowitz1988handbook}
M.~Abramowitz, I.~A. Stegun, and R.~H. Romer.
\newblock Handbook of mathematical functions with formulas, graphs, and
  mathematical tables, 1988.

\bibitem{askey1975orthogonal}
R.~Askey.
\newblock {\em Orthogonal polynomials and special functions}.
\newblock SIAM, 1975.

\bibitem{brafman1951generating}
F.~Brafman.
\newblock Generating functions of jacobi and related polynomials.
\newblock {\em Proceedings of the American Mathematical Society},
  2(6):942--949, 1951.

\bibitem{carlitz1990some}
L.~Carlitz and H.~Srivastava.
\newblock Some new generating functions for the hermite polynomials.
\newblock {\em Journal of Mathematical Analysis and Applications},
  149(2):513--520, 1990.

\bibitem{cohl2020lectures}
H.~S. Cohl and M.~E. Ismail.
\newblock {\em Lectures on Orthogonal Polynomials and Special Functions},
  volume 464.
\newblock Cambridge University Press, 2020.

\bibitem{el2006special}
R.~El~Attar.
\newblock {\em Special functions and orthogonal polynomials}, volume~3.
\newblock Lulu. com, 2006.

\bibitem{fan2015generating}
H.-Y. Fan, P.-F. Zhang, and Z.~Wang.
\newblock Generating function of product of bivariate hermite polynomials and
  their applications in studying quantum optical states.
\newblock {\em Chinese Physics B}, 24(5):050303, 2015.

\bibitem{kruchinin2021method}
D.~Kruchinin, V.~Kruchinin, and Y.~Shablya.
\newblock Method for obtaining coefficients of powers of bivariate generating
  functions.
\newblock {\em Mathematics}, 9(4):428, 2021.

\bibitem{thangavelu1993lectures}
S.~Thangavelu.
\newblock {\em Lectures on Hermite and Laguerre expansions}, volume~42.
\newblock Princeton University Press, 1993.

\bibitem{weber2003essential}
H.~J. Weber and G.~B. Arfken.
\newblock {\em Essential mathematical methods for physicists, ISE}.
\newblock Elsevier, 2003.

\bibitem{wunsche2001hermite}
A.~W{\"u}nsche.
\newblock Hermite and laguerre 2d polynomials.
\newblock {\em Journal of computational and applied mathematics},
  133(1-2):665--678, 2001.

\end{thebibliography}

\section*{Declarations}

\begin{itemize}

\item \textup{The author confirms sole responsibility for the study conception and manuscript preparation.}
\item \textup{No funding was received for conducting this study.}
\item \textup{The author declares that he has no conflict of interest.}
\item \textup{Author's affiliation: Manouchehr Amiri, Tandis Hospital, Tehran, Iran}
\end{itemize}

\end{document}